\documentclass[11pt]{amsart}
\usepackage[english]{babel}
\usepackage{amssymb,amsmath,xcolor,mathrsfs,mathabx}
\usepackage{hyperref,color}
\usepackage{graphicx}
\usepackage{wrapfig}
\usepackage[utf8]{inputenc}
\usepackage{float}
\usepackage{wrapfig}
\usepackage[all]{xy}
\usepackage{fancyhdr}
 
\numberwithin{equation}{section}
 
\newtheorem{theorem}{Theorem}[section]

\newtheorem{proposition}[theorem]{Proposition}

\newcommand{\Ric}{{\rm{Ric}}}

\newcommand{\Hess}{{\rm{Hess\,}}}

\title[Stekloff eigenvalues problems]{The first Stekloff eigenvalue in weighted Riemannian manifolds}

\author{M. Batista, J. I. Santos}
\address{IM, Universidade Fe\-deral de Alagoas, Macei\'o, 
AL, CEP 57072-970, Brazil}
\email{mhbs@mat.ufal.br \mbox{and} jissivan@gmail.com}

\thanks{The first author was supported by CNPq/Brazil. }

\begin{document} 

\begin{abstract}
In study of eigenvalue problems, a classical problem is the Stekloff eigenvalue problem. There are many estimates of the first non- zero Stekloff eigenvalue, including a sharp estimate on surfaces, obtained by Escobar in \cite{escobar}. In this paper we are interested to study this problem in weighted context. Various estimates are obtained, including a sharp estimate on surfaces, similar to that Escobar obtained.
\end{abstract}
\maketitle
\section{introduction}
The Classical Stekloff  problem in its original form is the eigenvalue problem 
\begin{align*}
& \begin{cases} \Delta u=0\,\,\,\,&\text{ in } \Omega ,\\
 \frac{\partial u}{\partial\nu}=\sigma u\,\,\,\,&\text{ on } \partial\Omega,
  \end{cases}
  \end{align*}
was introduced by him in \cite{stekloff} for bounded domains $\Omega$ of the plane and afterward this was studied by Payne in \cite{payne} for bounded domains in the plane with non-negative curvature. This problem has a physical interesting because the eigenfunctions represents the steady state temperature on a domain and the flux on the boundary is proportional to the temperature, see  \cite{stekloff} for more details. Thenceforth many authors have studied about  this subject and  they had many advances on the understanding of this topic, see for instance \cite{stekloff, weinstock, payne, escobar, escobar2, escobar 3, kuttler, lima, raulot, xia, wang, binoy, GP} and references therein. More specifically, many authors studied ways to estimate or determine exactly the eigenvalues associated with the Stekloff problem and modifications of the latter, see \cite{escobar, xia, wang}.

Another interesting object in Differential Geometry are the Riemannian manifolds endowed with a smooth positive density function, this is directly related to Ricci flow, mean curvature flow, theory of optimal transportation, see \cite{morgan, espinar} for a good overview of this subject.  Aiming study the various Stekloff's problems in the weighted context, we will introduce the necessary concepts. We recall that a {weighted Riemannian manifold} is a  Riemannian manifold $({M},g)$ 
endowed with a real-valued smooth function $f: M \to \mathbb{R}$ 
which is used as a density to measure geometric objects on $M$. 
Associated to this structure we have an important second order differential operator defined by
$$\Delta_f u= \Delta u - \langle\nabla u, \nabla f\rangle,$$
where $u \in C^\infty$. This operator is known as Drift Laplacian.

Also, following  Lichnerowich \cite{Lichnerowich} and Bakry and \'Emery  \cite{bakry}, 
the natural generalizations of Ricci curvatures  are defined as 
\begin{equation}\label{BE}
\Ric_f=\Ric + \Hess f
\end{equation}
and
\begin{equation}\label{mBE}
\Ric_f^k=\Ric_f-\frac{df\otimes df}{k-n-1},
\end{equation}
where $k>n+1$ or $k=n+1$ and $f$ a constant function.

\medskip

In this paper we will consider $M^{n+1}$ a compact oriented Riemannian manifold with boundary $\partial M$. Let $i: \partial M \hookrightarrow M$ be the standard inclusion and $\nu$ the outward unit normal on $\partial M$. We will denote by $II$ its {second fundamental form} associate to $\nu$, $\langle \nabla_X \nu, Y\rangle = II(X,Y),$
and by $H$ the {mean curvature} of $\partial M$, that is, the trace of $II$ over $n$. 

We recall that the {weighted mean
curvature}, introduced by Gromov in \cite{g}, of the inclusion $i$ is given by 
$$H_{f}=H-\dfrac{1}{n}\langle \nu,\nabla f\rangle.$$

%
 Finally, consider the following  three kinds of the weighted Stekloff problems:
\begin{align}\label{steklov}
& \begin{cases} \Delta_fu=0\,\,\,\,&\text{ in } M ,\\
 \frac{\partial u}{\partial\nu}=p u\,\,\,\,&\text{ on } \partial M;
  \end{cases}
  \intertext{}
  &  \begin{cases}\label{2}
  \Delta_f^2u=0\,\,\,\,&\text{ in } M,\\
  u=\Delta_fu-q\frac{\partial u}{\partial\nu}=0\,\,\,\,&\text{ on } \partial M;
  \end{cases}
   \intertext{}
   & \begin{cases}\label{15}
   \Delta_f^2u=0\,\,\,\,&\text{ in } M,\\
   u=\frac{\partial^2u}{\partial \nu^2}-q\frac{\partial u}{\partial \nu}=0\,\,\,\,&\text{ on } \partial M,
  \end{cases}
\end{align}
where $\nu$ denotes the outward unit normal on $\partial M$. The first non-zero eigenvalues
of the above problems will be denoted by  $p_1$ and $q_1$, respectively. We will use the same letter for the first non-zero eigenvalues of last two problems because whenever the weighted mean curvature of $\partial M$ is constant then the problems are equivalents. 

Lastly, for the sake of simplicity, we will omit the weighted volume element in the integrals in all text. Now, we are able to introduce our results. 

\medskip
 Our first result reads as follows:

\begin{theorem}\label{3}
Let $M^{n+1}$ be a compact weighted Riemannian manifold with $\Ric_f^k\geq0$ and boundary 
  $\partial M$.  Assume that the weighted 
mean curvature of $\partial M$ satisfies $H_f\geq\frac{(k-1)c}{n}$, to some positive constant $c$, and that second fundamental form 
$II\geq cI$ , in the quadratic form sense. Denote by $\lambda_1$ the first non-zero eigenvalue of the Drift
Laplacian acting on functions on $\partial M$. Let $p_1$ be  the first eigenvalue of the 
weighted Stekloff eigenvalue problem $(\ref{steklov})$. 
  Then,
\begin{equation}
 p_1\leq \frac{\sqrt{\lambda_1}}{(k-1)c}(\sqrt{\lambda_1}+\sqrt{\lambda_1-(k-1)c^2})
\end{equation}
with equality occurs if and only if $M$ is isometric to an $n$-dimensional euclidean ball of 
radius $\frac{1}{c}$, $f$ is constant and $k=n+1$.
\end{theorem}


The second result is the following:

\begin{theorem}\label{12}
Let $M^{n+1}$ be a compact connected weighted Riemannian manifold with $\Ric_f^k\geq0$ and boundary 
 $\partial M$. Assume that the weighted mean curvature of $\partial M$  satisfies $H_f\geq \frac{k-1}{k}c$, to some positive constant $c$. Let $q_1$ be the first eigenvalue of the 
weighted Stekloff eigenvalue problem $(\ref{2})$. Then
$$q_1\geq nc.$$ 
Moreover, equality occurs if  and only if $M$ is isometric to a euclidean ball of radius $\frac{1}{c}$ in 
$\mathbb R^{n+1}$, $f$ is constant and $k=n+1$.
\end{theorem}

The next results are

\begin{theorem}\label{17} 
Let $M^{n+1}$ be a compact connected weighted Riemannian manifold with boundary 
$\partial M$. Denote by $A,\,V$ the weighted area of $\partial M$ and the weighted volume of $M$, respectively.  Let $q_1$ be the first eigenvalue of the 
weighted Stekloff eigenvalue problem $(\ref{2})$. Then,
$$q_1\leq \frac{A}{V}.$$ 
Moreover, if in addition that the $\Ric_f^k$ of $M$ is non-negative and that there is a point $x_0\in \partial M$ such that $H_f(x_0)\geq\frac{A}{(n+1)V}$, and $q_1=\frac{A}{V}$ implies that $M$ is isometric to an $(n+1)$-dimensional Euclidean ball, $f$ is constant and $k=n+1$.
\end{theorem}

and

\begin{theorem}\label{191}
  Let $M^{n+1}$ be a compact connected weighted Riemannian manifold with 
$\Ric_f^k\geq0$ and boundary $\partial M$ nonempty. 
  Assume that $H_f\geq \frac{(k-1)c}{n}$, for some positive constant $c$. Let $q_1$ be the 
first eigenvalue of the problem $(\ref{15})$. Then $$q_1\geq c.$$ Moreover, equality 
occurs if  and only if $M$ is isometric to a ball of radius  $\frac{1}{c}$ in  $\mathbb R^{n+1}$, $f$ is constant and $k=n+1$.
 \end{theorem}

Lastly, we announce a sharp estimate of the first non-zero Stekloff eigenvalue of surfaces about suitable hypotheses.

\begin{theorem}\label{Escobar}
Let $M^2$ be a compact weighted Riemannian manifold with boundary. Assume that 
$M$ has non-negative $\Ric_f$, and that the  geodesic curvature of
$\partial M,\,k_g$ satisfies $k_g-f_{\nu}\geq c>0$. Let $p_1$ be the first
non-zero eigenvalue of the Stekloff problem $(\ref{steklov})$.  Assume that $f$ is constant on
the boundary $\partial M$, then $p_1\geq c.$ Moreover, the equality occur if and only if
$M$ is the Euclidean ball of radius $c^{-1}$ and $f$ is constant.
\end{theorem}

\section{Preliminaries}

In this section we recall some necessary results to prove the theorems enunciated in the introduction. We will present some proofs for the sake of completeness.

\medskip
In \cite{batista} the authors proved the following useful inequality. 
\begin{proposition}\label{9}
 Let $u$ be a smooth function on $M^{n+1}$. then we have 
 \begin{equation*}
  |\Hess u|^2+\Ric_f(\nabla u,\nabla u)\geq \frac{(\Delta_f u)^2}{k}+\Ric_f^k(\nabla u,\nabla u),
 \end{equation*}
for every $k>n+1$ or $k=n+1$ and $f$ is a constant. Moreover, equality holds if and only if $\Hess u=\frac{\Delta u}{n+1} \langle\,,\rangle$ and 
$\langle\nabla u,\nabla f\rangle=-\frac{k-n-1}{k}\Delta_f u \footnote{This term only appear in the case of a  non constant function.}$.
\end{proposition}
\begin{proof}
 Let $\{e_1,\ldots,e_{n+1}\}$ be a orthonormal basis of $T_pM$, then by Cauchy-Schwarz inequality
 we have that
 \begin{align}\label{10}
  (\Delta u)^2 \leq(n+1)|\Hess u|^2.
 \end{align}
Using that $\frac{1}{n+1}a^2+\frac{1}{k-n-1}b^2\geq\frac{1}{k}(a-b)^2$ with equality if and only if 
\begin{equation}\label{11}
a=-\frac{(n+1)b}{k-n-1}, 
\end{equation}

we obtain
\begin{align}
|\Hess u|^2+\Ric_f(\nabla u,\nabla u)&\geq \frac{1}{n+1}(\Delta u)^2+ \Ric_f^k(\nabla u,\nabla u)+
\frac{\langle\nabla f,\nabla u\rangle^2 }{k-n-1}\nonumber\\
&\geq\frac{1}{k}(\Delta u-\langle\nabla f,\nabla u\rangle)^2+\Ric_f^k(\nabla u,\nabla u)\\
&=\frac{1}{k}(\Delta_f u)^2+\Ric_f^k(\nabla u,\nabla u).\nonumber
\end{align}
If the equality holds, then since we use the Cauchy-Schwarz's inequality in $(\ref{10})$ we obtain that
$\Hess u=\lambda\langle\,,\rangle$, and by $(\ref{11})$ 
$$\Delta u=-\frac{(n+1)\langle\nabla f,\nabla u\rangle}{k-n-1},$$
Consequently
 $$\Delta_fu=-\frac{(n+1)\langle\nabla 
f,\nabla u\rangle}{k-n-1}-\langle\nabla f,\nabla 
u\rangle=-\frac{k}{k-n-1}\langle\nabla f,\nabla u\rangle.$$
The converse is immediate.
\end{proof}

In \cite{lima} the authors showed that, for a smooth function $u$ defined on an $n$-dimensional compact weighted manifold $M$ with 
boundary $\partial M$, the following identity holds if $h=\frac{\partial u}{\partial \nu},\,z=u|_{\partial M}$ and $\Ric_f$
denotes the generalized Ricci curvature of $M$:
\begin{align}\label{1}
 \int_M[(\Delta_f u)^2-&|\Hess u|^2- \Ric_f(\nabla u,\nabla u)]=\\
&= \int_{\partial M}\left[nH_fh^2+2h\overline{\Delta}_fz+II(\overline{\nabla}z,\overline{\nabla}z)\right].\nonumber
\end{align}
 Here, $\overline{\Delta}$ and $\overline{\nabla}$ represent the Laplacian and the gradient on $\partial M$ with respect to the
 induced metric on $\partial M$, respectively. 
 

Using the Proposition \ref{9} we have that
 \begin{align}\label{6}
  \int_M\frac{k-1}{k}[(\Delta_fu)^2-&\Ric_f^k(\nabla u,\nabla u)]\geq\\
  &\geq\int_{\partial M}[nH_fh^2+2h\overline{\Delta}_fz+II(\overline{\nabla}z,\overline{\nabla}z)]\nonumber.
 \end{align}
 
 \medskip

The next result is an estimate for the first non-zero eigenvalue of the Drift Laplacian on closed submanifolds.
This result is a slight modification of Theorem 1.6 in \cite{huang} and it reads as follows.
\begin{proposition}\label{4}
 Let $M^{n+1}$ be a compact weighted Riemannian manifold 
with nonempty boundary $\partial M$ and $\Ric_f^k\geq0$. If the second 
fundamental form of $\partial M$ satisfies $II\geq cI$, in the quadratic form sense, and 
$H_f\geq\frac{k-1}{n}c$, then $$\lambda_1(\partial M)\geq (k-1)c^2,$$ where 
$\lambda_1$ is the first non-zero eigenvalue of the Drift Laplacian acting 
on functions on $\partial M$. The equality holds if and only if $M$ is 
isometric to an Euclidean ball of radius $\frac{1}{c}$, $f$ is constant and $k=n+1$.
\end{proposition}
\begin{proof}
 Let $z$ be an eigenfunction corresponding to the first non-zero eigenvalue $\lambda_1$ of 
the Drift Laplacian of $\partial M$, that is, 
\begin{equation}
 \overline{\Delta}_fz+\lambda_1z=0.
\end{equation}
 Let $u\in C^{\infty}(M)$ be the solution of the Dirichlet problem
$$\begin{cases}
 \Delta_fu=0\,\,\,\,\,&\text{ in } M,\\
 u=z\,\,\,\,\,&\text{ on }\partial M.
\end{cases}
$$
It then follows from $(\ref{6})$ and the non-negativity of $\Ric_f^k$ of $M$ that
\begin{align}
 0&\geq  \int_{\partial M}[nH_fh^2+2h\overline{\Delta}_fz+II(\overline{\nabla}z,\overline{\nabla}z)].
\end{align}
Since $II\geq cI$, we have 
$$II(\overline{\nabla}z,\overline{\nabla}z)\geq c |\overline{\nabla z}|^2,$$
and noticing that
$$\int_{\partial M}|\overline{\nabla z}|^2=-\int_{\partial M}z\overline{\Delta}z=\lambda_1\int_{\partial M}z^2,$$
we obtain
\begin{align*}
 0&\geq  \int_{\partial M}[nH_fh^2+2h\overline{\Delta}_fz+II(\overline{\nabla}z,\overline{\nabla}z)]\\
 &\geq\int_{\partial M}[(k-1)ch^2-2\lambda_1zh+c\lambda_1z^2]\\
 &=\int_{\partial M}\bigg[(k-1)c\left(h-\frac{\lambda_1z}{(k-1)c}\right)^2+\lambda_1\left(c-\frac{\lambda_1}{(k-1)c}\right)z^2\bigg]\\
 &\geq\lambda_1\left(c-\frac{\lambda_1}{(k-1)c}\right)\int_{\partial M}z^2.
\end{align*}
Consequently, $$\lambda_1\geq(k-1)c^2,$$
which proof the first part of theorem. The equality case follows by Proposition \ref{9} and a careful analysis in the equalities that occur. The converse is immediate.


\end{proof}

Recall the following version of Hopf boundary point lemma, see its proof  in \cite{gilbarg}, Lemma 3.4.
\begin{proposition}[Hopf boundary point lemma]\label{25}
 Let $(M^n,g)$ be a complete Riemannian manifold and let $\Omega\subset M$ be a closed domain. If $u:\Omega\rightarrow\mathbb R$ 
 is a function with $u\in C^2(\text{int}(\Omega))$ satisfying
 $$\Delta u+\langle X,\nabla u\rangle\geq0,$$
 where $X$ is a bounded vector field, $x_0\in\partial\Omega$ is a point where 
 $$u(x)<u(x_0)\,\,\,\forall x\in \Omega,$$
 $u$ is continuous at $x_0$, and $\Omega$ satisfies the interior sphere condition at $x_0$, then 
 $$\frac{\partial u}{\partial \nu}(x_0)>0$$
 if this outward normal derivative exists.
\end{proposition}

%

\section{proof of the sharp bounds for the Stekloff eigenvalues}

In this section we will give the proof of the four first results announced in introduction and for this we will use all tools presented in the preliminaries.

\medskip

\noindent{ \bf Proof of Theorem \ref{3}.} Let $u$ be the solution of the following problem
$$\begin{cases}
 \Delta_fu=0\,\,\,\,&\text{in } M,\\
 u|_{\partial M}=z,&
\end{cases}
$$
where $z$ is a first eigenfunction on $\partial M$ corresponding to $\lambda_1$, that is, $z$ satisfies $\overline{\Delta}_fz+\lambda_1z=0$ on $\partial M$. Set 
$h=\frac{\partial u}{\partial \nu}\big|_{\partial M}$, then we have from the Rayleigh inequality that (cf.\cite{kuttler})
\begin{align}
 p_1&\leq\frac{\int_{\partial M} h^2}{\int_M|\nabla u|^2}\label{7}
 \intertext{and}
 p_1&\leq\frac{\int_M|\nabla u|^2}{\int_{\partial M} z^2}\label{8}
 \end{align}
Notice that $(\ref{8})$ it is the variational principle, and $(\ref{7})$ it is obtained as follows, 
\begin{align*}
 p_1&\leq \frac{\int_M|\nabla u|^2}{\int_{\partial M}z^2}=\frac{-\int_M u\Delta_fu+
 \int_{\partial M} u\langle\nabla u,\nu\rangle}{\int_{\partial M}z^2}\\
 &=\frac{\int_M|\nabla u|^2}{\int_{\partial M}z^2}\cdot\frac{\int_{\partial M} u\langle\nabla u,\nu\rangle}{\int_M|\nabla u|^2}\\
 &=\frac{1}{\int_{\partial M}z^2}\cdot\frac{\left(\int_{\partial M} u\langle\nabla u,\nu\rangle\right)^2}{\int_M|\nabla u|^2}\\
 &\leq \frac{\int_{\partial M}z^2}{\int_{\partial M}z^2}\cdot\frac{\int_{\partial M}
 \langle\nabla u,\nu\rangle^2}{\int_M|\nabla u|^2}\\
 &=\frac{\int_{\partial M} h^2}{\int_M|\nabla u|^2},
\end{align*}
 which gives
 \begin{equation}\label{5} 
 p_1^2\leq\frac{ \int_{\partial M} h^2}{ \int_{\partial M} z^2}.
 \end{equation}

It then follows by substituting $u$ into the equation $(\ref{1})$, and using 
the { Proposition \ref{9}}, that 
 \begin{align} \label{351}
  0\geq\int_M\frac{k-1}{k}[(\Delta_fu)^2-&\Ric_f^k(\nabla u,\nabla u)]\geq\\
  &\geq\int_{\partial M}[nH_fh^2+2h\overline{\Delta}_fz+II(\overline{\nabla}z,\overline{\nabla}z)]\nonumber\\
  &\geq\int_{\partial M}[(k-1)ch^2-2\lambda_1z+c|\overline{\nabla}z|^2].\nonumber
 \end{align}
Note that, by Green's formula, 
$$\int_{\partial M}|\overline{\nabla}z|^2=\int_{\partial M}\langle \overline{\nabla}z,\overline{\nabla}z\rangle=
-\int_{\partial M}z\overline{\Delta}_fz=\lambda_1\int_{\partial M}z^2.$$
Putting this expression in (\ref{351}) we have that
\begin{align*}
 0 &\geq(k-1)c\int_{\partial M}h^2-2\lambda_1\int_{\partial M}hz+c\lambda_1\int_{\partial M}z^2\\
 &\geq (k-1)c\int_{\partial M}h^2-2\lambda_1\left(\int_{\partial M}h^2\right)^{\frac{1}{2}}
 \left(\int_{\partial M}z^2\right)^{\frac{1}{2}}+c\lambda_1\int_{\partial M}z^2\\
 &=\frac{(k-1)c^2-\lambda_1}{c}\int_{\partial M}h^2+\left[\sqrt{\frac{\lambda_1}{c}}\left(\int_{\partial M}h^2\right)^{\frac{1}{2}}-
 \sqrt{c\lambda_1}\left(\int_{\partial M}z^2\right)^{\frac{1}{2}}\right]^2,
\end{align*}
from where

$$\frac{\sqrt{\lambda_1-(k-1)c^2}}{\sqrt{c}}\left(\int_{\partial M}h^2\right)^{\frac{1}{2}}\geq\sqrt{\frac{\lambda_1}{c}}
\left(\int_{\partial M}h^2\right)^{\frac{1}{2}}-\sqrt{c\lambda_1}\left(\int_{\partial M}z^2\right)^{\frac{1}{2}}$$
and
$$\frac{\sqrt{\lambda_1}-\sqrt{\lambda_1-(k-1)c^2}}{\sqrt{c}}\left(\int_{\partial M}h^2\right)^{\frac{1}{2}}\leq
\sqrt{c\lambda_1}\left(\int_{\partial M}z^2\right)^{\frac{1}{2}},$$
that is,
\begin{align*}
\left(\int_{\partial M}h^2\right)^{\frac{1}{2}}&\leq \frac{c\sqrt{\lambda_1}}{\sqrt{\lambda_1}-
\sqrt{\lambda_1-(k-1)c^2}}\left(\int_{\partial M}z^2\right)^{\frac{1}{2}}\\
&=\frac{\sqrt{\lambda_1}}{(k-1)c}(\sqrt{\lambda_1}+\sqrt{\lambda_1-(k-1)c^2})\left(\int_{\partial M}z^2\right)^{\frac{1}{2}}.
\end{align*}
Using $(\ref{5})$, we obtain
\begin{align*}
p_1&\leq \frac{\sqrt{\lambda_1}}{(k-1)c}\left(\sqrt{\lambda_1}+\sqrt{\lambda_1-(k-1)c^2}\right).\\
\end{align*}

Now, assume that
 $$p_1= \frac{\sqrt{\lambda_1}}{(k-1)c}\left(\sqrt{\lambda_1}+\sqrt{\lambda_1-(k-1)c^2}\right).$$
So, we also have that
$$\left(\int_{\partial M}h^2\right)^{\frac{1}{2}}= \frac{\sqrt{\lambda_1}}{(k-1)c}\left(\sqrt{\lambda_1}+\sqrt{\lambda_1-(k-1)c^2}\right)\left(\int_{\partial M}z^2\right)^{\frac{1}{2}}\vspace{3mm}$$
and all inequalities above become equality. Thus $h=\alpha z$ and 
$$\alpha=\frac{\left(\alpha^2\int_{\partial M}z^2\right)^{\frac{1}{2}}}{\left(\int_{\partial M}z^2\right)^{\frac{1}{2}}}=
 \frac{\sqrt{\lambda_1}}{(k-1)c}\left(\sqrt{\lambda_1}+\sqrt{\lambda_1-(k-1)c^2}\right),$$
that is, 
$$h= \frac{\sqrt{\lambda_1}}{(k-1)c}(\sqrt{\lambda_1}+\sqrt{\lambda_1-(k-1)c^2})z.$$
Furthermore we infer, by Proposition \ref{9},  that $\Hess u=0$. Now, on the boundary $\partial M$, we can write 
\begin{align*}\nabla 
u&=(\nabla u)^\top+(\nabla u)^{\perp}\\
&=(\nabla u)^\top+\langle\nabla u,\nu\rangle\nu,
\end{align*}
where 
$(\nabla u)^\top$ is tangent to $\partial M$ and $(\nabla 
u)^{\perp}$ is normal to $\partial M$. Then, take a local 
orthonormal fields $\{e_i\}_{i=1}^{n}$ tangent to $\partial M$. We obtain
\begin{align*}
 0&=\sum_{i=1}^{n}\Hess u(e_i,e_i)=\sum_{i=1}^{n}\langle\nabla_{e_i}\nabla 
u,e_i\rangle\\
&=\sum_{i=1}^{n}\langle\nabla_{e_i}[(\nabla u)^\top+\langle\nabla u,\nu\rangle\nu],e_i\rangle\\
&=\sum_{i=1}^{n}\langle\nabla_{e_i}(\nabla u)^\top+\langle\nabla u,\nu\rangle\nabla_{e_i}\nu+e_i(\langle\nabla 
u,\nu\rangle)\nu,e_i\rangle\\
&=\overline{\Delta}z+\sum_{i=1}^{n}\langle\nabla 
u,\nu\rangle \, II(e_i,e_i)\\
&=\overline{\Delta}z+nHh\\
&=\overline{\Delta}_fz - f_\nu h+nHh\\
&==\overline{\Delta}_fz +nH_fh\\
&=-\lambda_1z+c(k-1)h\\
&=-\lambda_1z+c(k-1) \frac{\sqrt{\lambda_1}}{(k-1)c}(\sqrt{\lambda_1}+\sqrt{\lambda_1-(k-1)c^2})z,
\end{align*}
from where
$$\lambda_1=(k-1)c^2.$$

Therefore, follows by Proposition \ref{4} that $M$ is isometric to an $(n+1)$-dimensional Euclidean ball of radius $\frac{1}{c}$, $f$ is constant and so $k=n+1$. The converse follows the ideas of the Riemannian case.
\qed
\medskip

{\bf Proof of Theorem \ref{12}.} Let $w$ be an eigenfunction corresponding to 
the first eigenvalue $q_1$ of problem $(\ref{2})$,
that is,
 \begin{equation}
 \begin{cases}
   \Delta_f^2w=0\,\,\,\,\,\,&\text{ in } M,\\
  w=\Delta_fw-q_1 \frac{\partial w}{\partial \nu}=0\,\,\,\,\,\, &\text{ on } \partial M.
  \end{cases}
 \end{equation}
 Set $\eta=\frac{\partial w}{\partial\nu}|_{\partial M}$; then
 by divergence theorem we obtain
 \begin{align*}
  \int_M(\Delta_fw)^2&=-\int_{M}\langle\nabla(\Delta_fw),\nabla w\rangle+\int_{\partial M}\Delta_fw\, \langle\nabla w,\nu\rangle\\
  &=\int_{M}w\, \Delta_f(\Delta_fw)-\int_{\partial M}w\, \langle\nabla(\Delta_fw),\nu\rangle+\int_{\partial M}\Delta_fw \,\langle\nabla w,
  \nu\rangle\\
  &=q_1\int_{\partial M}\eta^2,
 \end{align*}
that is, 
$$q_1=\frac{ \int_M(\Delta_fw)^2}{\int_{\partial M}\eta^2}.$$
 
Substituting $w$ in $(\ref{6})$, and noting that $w|_{\partial M}=z$, we have
\begin{align*}
 \frac{k-1}{k}\int_{M}(\Delta_fw)^2&\geq\int_M\Ric_f^k(\nabla w,\nabla w)+\int_{\partial M}nH_f\eta^2\nonumber\\
 &\geq \frac{(k-1)nc}{k}\int_{\partial M}\eta^2,
\end{align*}
from where $q_1\geq nc,$ as we desired.\vspace{2mm}

Assume now that $q_1=nc$, then the inequalities above become equalities
and consequently $H_f=\frac{k-1}{k}c$. 
Furthermore, we have equality in the Proposition \ref{9}, 
thus $\Hess w=\frac{\Delta w}{n+1}\langle\,,\rangle$ and
$\Delta_fw=\frac{k}{n+1}\Delta w$.

 Take an orthonormal frame $\{e_1,\ldots,e_n,e_{n+1}\}$ on $M$ 
such that when restricted to $\partial M,\,e_{n+1}=\nu$. Since $w|_{\partial M}=0$ we have
\begin{align*}
e_i(\eta)&=e_i\langle\nabla w,\nu\rangle\\
&=\langle\nabla_{e_i}\nabla w,\nu\rangle+\langle\nabla w,\nabla_{e_i}\nu\rangle\\
&=\Hess w(e_i,\nu)+II((\nabla w)^\top,e_i)=0,
\end{align*}
that is, $\eta=\rho=$ constant, and so $(\Delta_fw)|_{\partial M}=q_1\eta=nc\rho$ 
is also a constant. Using the fact that $\Delta_fw$ is a $f$-harmonic
function on $M$, we conclude by maximum principle that $\Delta_fw$ is constant on $M$. 
Since $\Delta_fw=\frac{k}{n+1}\Delta w$, then $w$ satisfies 
$$\begin{cases}
   \Hess w=\frac{\Delta_f w}{k}\langle\,,\rangle\hspace{3mm} \text{in} \hspace{3mm} M,\\
   w|_{\partial M}=0.
  \end{cases}$$
  
 Thus, by Lema 3 in \cite{robert}, we conclude that $M$ is isometric to a ball in $\mathbb R^{n+1}$ of radius $c^{-1}.$ Now, using the hessian of $w$ is possible see that $w=\frac{\lambda}{2}r^2 + C,$ where $\lambda=\frac{\Delta_f w}{k}$ and $r$ is the distance function from its minimal point, see \cite{robert} for more details for this technique.

Lastly, we will show that $f$ is constant. In fact, if $k> n+1$, then $\langle\nabla f,\nabla w\rangle$ is constant and so $f=-(k-n-1)\ln r +C$. It is a contradiction, since $f$ is a smooth function.
  
\qed 
 
 \bigskip
\noindent{\bf Proof of Theorem \ref{17}.} 
Now, let $w$ be the solution of the following Drift Laplace equation
\begin{equation}
 \begin{cases}
  \Delta_f w=1\,\,\,\,\text{ in } M,\\
  w|_{\partial M}=0.
 \end{cases}
\end{equation}
Follows from Rayleigh characterization of $q_1$ that
\begin{equation}
 q_1\leq \frac{\int_M|\nabla w|^2}{\int_{\partial M}\eta^2}=\frac{\int_M(\Delta_fw)^2}{\int_{\partial M}\eta^2}=
 \frac{V}{\int_{\partial M}\eta^2},
\end{equation}
where $\eta=\frac{\partial w}{\partial\nu}\big|_{\partial M}$. Integrating $\Delta_f w=1$ on $M$ and using the divergence theorem,
it gives 
$$V=\int_{\partial M}\eta.$$
 Hence we infer from Schwarz inequality that 
 \begin{equation}\label{18}
 V^2\leq A\int_{\partial M}\eta^2. 
 \end{equation}

 Consequently, 
 $$q_1\leq\frac{V}{\int_{\partial M}\eta^2}\leq\frac{V}{V^2/A}=\frac{A}{V}.$$
 
 Assume now that $\Ric_f^k\geq0$, $H_f(x_0)\geq\frac{(k-1)A}{k\, n\, V}$ for some $x_0\in \partial M$ and $q_1=\frac{A}{V}$. In this
 case $(\ref{18})$ become a equality and so $\eta=\frac{V}{A}$ is a constant. Consider the function $\phi$ on $M$ given by
 $$\phi=\frac{1}{2}|\nabla w|^2-\frac{w}{k}.$$
 Using the Bochner formula $(\ref{16})$, $\Delta_fw=1$, the {\ Proposition \ref{9}} and that $\Ric_f^k\geq0$, we have that
 \begin{align}\label{28} 
\frac{1}{2}\Delta_f\phi&=|\Hess w|^2+\langle \nabla w,\nabla(\Delta_f w)\rangle+\Ric_f(\nabla w,\nabla w)-\frac{1}{k}\\
&\geq\frac{1}{k}(\Delta_fw)^2-\frac{1}{k}=0.\nonumber
\end{align}
Thus $\phi$ is $f$-subharmonic. Observe that $\phi=\frac{1}{2}\left(\frac{V}{A}\right)^2$ on the boundary. In fact, if we write $\nabla w=(\nabla w)^\top+(\nabla w)^{\perp}$, where $(\nabla w)^\top$ is tangent to $\partial M$ and $(\nabla w)^{\perp}$ is normal to $\partial M$, and since $w|_{\partial M}=0$, it follows that $\nabla w=(\nabla w)^{\perp}=C\nu$ on $\partial M$. On the other hand,
$$1=\Delta_fw=q_1\langle\nabla w,\nu\rangle=\frac{A}{V}C\,\,\,\mbox{implies} \,\,\,C=\frac{V}{A}\,\,\,\text{ and }\,\,\,|\nabla w|=\frac{V}{A}.$$
Therefore $\phi=\frac{1}{2}\left(\frac{V}{A}\right)^2$ on the boundary, and so we conclude by Proposition \ref{25} that either
\begin{align}
 \phi=\frac{1}{2}\left(\frac{V}{A}\right)^2\,\,\,\,\,\text{ in } M\label{26}
 \intertext{or}
 \frac{\partial\phi}{\partial\nu}(y)>0,\,\,\,\,\, \forall \, y\in\partial M.\label{27}
\end{align}
From $w|_{\partial M}=0$, we have
\begin{align*}
 1=(\Delta_fw)|_{\partial M}&=nH\eta+\Hess w(\nu,\nu)-\frac{V}{A}\langle \nabla f,\nu\rangle\\
 &=\frac{nV}{A}\left(H_f+\frac{\langle\nabla f,\nu\rangle}{n}\right)+\Hess w(\nu,\nu)-\frac{V}{A}\langle \nabla f,\nu\rangle\\
 &=\frac{nV}{A}H_f+\Hess w(\nu,\nu).
\end{align*}
Hence it holds on $\partial M$ that 
\begin{align*}
 \frac{\partial\phi}{\partial\nu}&=\frac{V}{A}\Hess w(\nu,\nu)-\frac{V}{k\, A}\\
 &=\frac{V}{A}\left(1-\frac{nV}{A}H_f\right)-\frac{V}{k\, A}\\
 &=n\frac{V}{A}\left(\frac{k-1}{k\, n}-H_f\frac{V}{A}\right),
\end{align*}
which shows that $(\ref{27})$ is not true since $H_f(x_0)\geq\frac{(k-1)A}{k\, n\, V}$. Therefore $\phi$ is constant on $M$. Since the Drift Laplacian of $\phi$ vanishes, we infer that equality must hold in $(\ref{28})$ and that give us equality in the { Proposition \ref{9}}, and consequently $1=\Delta_fw=\frac{k}{n+1}\Delta w$ and $\Hess w=\frac{\Delta w}{n+1}\langle\,,\rangle$.   The remainder of the proof follows a similar arguments  as in proof  of Theorem \ref{12}.

\medskip\qed

\noindent{\bf Proof of Theorem \ref{191}.} Let $w$ be an eigenfunction corresponding to the first eigenvalue $q_1$ of the problem $(\ref{15})$:
  \begin{equation} 
  \begin{cases}
   \Delta_f^2u=0\,\,\,\,&\text{ in } M,\\
   u=\frac{\partial^2u}{\partial \nu^2}-q\frac{\partial u}{\partial \nu}=0\,\,\,\,&\text{ on } \partial M.
  \end{cases}
 \end{equation}
 Observe that $w$ is not a constant. Otherwise, we would conclude from $w|_{\partial M}=0$ that $w\equiv0$. Set
 $\eta=\frac{\partial w}{\partial \nu}|_{\partial M}$; then $\eta\neq0$. In fact, if $\eta=0$ then
 $$w|_{\partial M}=(\nabla w)|_{\partial M}=\frac{\partial^2w}{\partial\nu^2}=0$$
 which implies that $(\Delta_fw)|_{\partial M}=0$ and so $\Delta_fw=0$ on $M$ by the maximum principal, which in turn implies 
 that $w=0$. This is a contradiction.
\medskip

Since $w|_{\partial M}=0$, we have by the divergence theorem that
\begin{equation}
\int_M\langle\nabla w,\nabla(\Delta_fw)\rangle=-\int_Mw\Delta_f^2w=0,
\end{equation}
hence
\begin{equation}\label{36}
 \int_{\partial M}\Delta_fw\, \frac{\partial w}{\partial 
\nu}=\int_M\langle\nabla(\Delta_f w),\nabla 
w\rangle+\int_M(\Delta_fw)^2=\int_M(\Delta_fw)^2.
\end{equation}
Since $w|_{\partial M}=0$, we have $\nabla w=\frac{\partial w}{\partial \nu}\nu$ and
\begin{align}\label{39}
 (\Delta_fw)|_{\partial 
M}&=\frac{\partial^2w}{\partial\nu^2}+nH\frac{\partial 
w}{\partial\nu}-\langle\nabla f,\nabla w\rangle\\
&=q_1\frac{\partial w}{\partial\nu}+nH_f\frac{\partial 
w}{\partial\nu}+\langle\nabla f,\nu\rangle\frac{\partial 
w}{\partial\nu}-\langle\nabla f,\nu\rangle\frac{\partial 
w}{\partial\nu}\nonumber\\
&=q_1\frac{\partial w}{\partial\nu}+nH_f\frac{\partial 
w}{\partial\nu}\nonumber
\end{align}
using $(\ref{36})$ and $(\ref{39})$ we obtain that
\begin{align*}
 q_1&=\frac{\int_M(\Delta_f w)^2-n\int_{\partial 
M}H_f\eta^2}{\int_{\partial M}\eta^2}.
\end{align*}
On the other hand, substituting $w$ into $(\ref{6})$, we obtain
\begin{align}\label{38}
\frac{k-1}{k}\int_M(\Delta_fw)^2&=\int_M\Ric_f^k(\nabla 
w,\nabla w)+\int_{\partial M}nH_f\eta^2\\
&\geq\int_{\partial M}nH_f\eta^2,\nonumber
\end{align}
that is,
$$\int_M(\Delta_fw)^2 - \int_{\partial M}nH_f\eta^2\geq \dfrac{n}{k-1}\int_{\partial M}H_f\eta^2\geq c\int_{\partial M}\eta^2.$$
By expression for $q_1$ and estimate above, we obtain the desired estimate
\begin{equation}\label{37}
 q_1\geq c.
\end{equation}
Assume now that $q_1=c$. So all inequalities in $(\ref{38})$  become equalities. Thus, by Proposition 
\ref{9}, we have that 
\begin{equation}\label{40}
 \Hess w=\frac{\Delta w}{n+1}\langle\,,\rangle\hspace{5mm}\text{ and }\hspace{5mm}\Delta_f w=-\frac{k}{k-n-1}\langle\nabla f,\nabla w\rangle.
\end{equation}
Choice an orthonormal frame $\{e_1,\ldots, e_n\}$ on $M$ so that restricted to $\partial M,\,e_n=\nu$. On the other side, to 
$i=1,\ldots,n-1$, using that $w|_{\partial M}=0$, we obtain
\begin{align*}
 0=\Hess w(e_i,e_n)&=e_ie_n(w)-\nabla_{e_i}e_n(w)\\
 &=e_i(\eta)-\langle\nabla_{e_i}e_n,e_n\rangle\eta=e_i(\eta),
\end{align*}
follow that $\eta=b_0=const.$ Since  $(\ref{37})$ takes equality and $\eta$ is constant, we conclude that $H_f = \frac{k-1}{n}c$, which implies 
from $(\ref{39})$ that $(\Delta_f w)|_{\partial M}=kcb_0$, therefore, by maximum principle $\Delta_fw$ is constant on $M$ which implies from 
$(\ref{40})$ that $\Delta w$ is constant on $M$. The remainder of the proof follows a similar arguments  as in proof  of Theorem \ref{12}.


\qed

\section{Escobar type theorem for the Stekloff problem}
 
Recall the Bochner type formula for weighted Riemannian manifold, which says: Any smooth function $u$ on $M$ holds that
\begin{equation}\label{16}
\frac{1}{2}\Delta_f|\nabla u|^2=|\Hess u|^2+\langle \nabla u,\nabla(\Delta_f u)\rangle+\Ric_f(\nabla u,\nabla u). 
\end{equation}

An immediate consequence of the Bochner type formula is the result below, however we believe that this is not a sharp estimate.

\begin{theorem}
 Let $M^{n+1},\,n\geq2$ be a compact weighted Riemannian manifold with boundary 
$\partial M$. Assume that $\Ric_f\geq 0, H_f\geq 0$ and that the second fundamental form 
satisfies $II\geq cI$ on $\partial M,\,c>0$. Then 
$$p_1>\frac{c}{2}.$$
\end{theorem}
\begin{proof}
 Set $h=\frac{\partial u}{\partial \nu}$, and $z=u|_{\partial M}$  where $u$ is solution of problem $(\ref{steklov})$. We have $p_1z=p_1u=h$, thus 
 $p_1\overline{\nabla}z=\overline{\nabla}h$. By $(\ref{1})$, we have
 \begin{align*}
  0>-\int_M|\Hess u|^2 &\geq\int_M[(\Delta_f u)^2-|\Hess u|^2- \Ric_f(\nabla u,\nabla u)]\\
&= \int_{\partial M}\left[nH_fh^2+2h\overline{\Delta}_fz+II(\overline{\nabla}z,\overline{\nabla}z)\right]\\
&\geq -2\int_{\partial M}\langle\overline{\nabla}h,\overline{\nabla}z\rangle+c\int_{\partial M}|\overline{\nabla}z|^2\\
&\geq -2p_1\int_{\partial M}|\overline{\nabla}z|^2+c\int_{\partial M}|\overline{\nabla}z|^2
 \end{align*}
Note that
$$\int_{\partial M}|\overline{\nabla}z|^2>0.$$
Otherwise $z$ is constant on the Boundary and hence $f$ is constant on $M$ which is a contradiction. Thus $p_1>\frac{c}{2}$.
\end{proof}


\bigskip

Below we present the proof of a sharp estimate of the non-zero first Stekloff eigenvalue on surfaces. The technique was introduced by Escobar in \cite{escobar}, and  just enable us to attack this problem in context of surfaces. 

\medskip

{\bf Proof of Theorem \ref{Escobar}.} 
Let $\phi$ be a non-constant eigenfunction for the Stekloff problem $(\ref{steklov})$. Consider the function $v=\frac{1}{2}|\nabla \phi|^2$,
 then by $(\ref{16})$
 \begin{align*}
  \Delta_fv=|\Hess \phi|^2+\langle \nabla \phi,\nabla(\Delta_f \phi)\rangle+\Ric_f(\nabla \phi,\nabla \phi).
 \end{align*}
Since $\phi$ is a $f$-harmonic function and $\Ric_f\geq0$ we find that 
\begin{equation}\label{33}
  \Delta_fv=|\Hess \phi|^2+\Ric_f(\nabla \phi,\nabla \phi)\geq0.
 \end{equation}
Therefore the maximum of $v$ is achieved
at some point $P\in\partial M$. The {Proposition \ref{25}} implies that
$(\partial v/\partial\eta)(P)>0$ or $v$ is identically constant.\vspace{5mm}\\
Let's assume  $(\partial v/\partial\eta)(P)>0$ and let $(t,x)$ be Fermi coordinates around the point $P$, that is, $x$ represents
a point on the curve $\partial M$ and $t$ represents the distance to the boundary point $x$. The metric  has the form
\begin{equation}\label{19}
 ds^2=dt^2+h^2(t,x)dx^2,
\end{equation}
where $h(P)=1,\,(\partial h/\partial x)(P)=0$. Thus
$$|\nabla \phi|^2=\left(\frac{\partial \phi}{\partial t}\right)^2+h^{-2}\left(\frac{\partial \phi}{\partial x}\right)^2,$$
and 
$$\frac{\partial v}{\partial x}=\frac{\partial \phi}{\partial t}\frac{\partial^2\phi }{\partial x\partial t}+h^{-2}
\frac{\partial \phi}{\partial x}\frac{\partial^2\phi }{\partial 
x^2}-h^{-3}\frac{\partial h}{\partial x}\left(\frac{\partial \phi}
{\partial x}\right)^2.$$
Evaluating at the point $P$ we obtain 
\begin{equation}\label{24} 
\frac{\partial v}{\partial x}(P)=\frac{\partial \phi}{\partial t}\frac{\partial^2\phi }{\partial x\partial t}+\frac{\partial \phi}
{\partial x}\frac{\partial^2\phi }{\partial x^2}=0.
\end{equation}
The $f$-Laplacian with respect to the metric given by $(\ref{19})$ in Fermi coordinates $(t,x)$ is
$$\Delta_f=\frac{\partial^2}{\partial t^2}+h^{-1}\frac{\partial h}{\partial t}\frac{\partial }{\partial t}+h^{-1}\frac{\partial }
{\partial x}\left(h^{-1}\frac{\partial }{\partial x}\right)-\frac{\partial f}{\partial t}\frac{\partial }{\partial t}-
h^{-2}\frac{\partial f}{\partial x}\frac{\partial }{\partial x}.$$
The geodesic curvature of $\partial M$ can be calculated in terms of the function $f$ and its first derivative as follows:
\begin{align}\label{20}
 k_g&=-\left\langle \nabla_{\partial/\partial x}\frac{\partial}{\partial t},\frac{\partial}{\partial x}\right\rangle=-\left\langle \nabla_{\partial/\partial t}\frac{\partial}{\partial x},\frac{\partial}{\partial x}\right\rangle\nonumber\\
 &=-\frac{1}{2}\frac{\partial}{\partial t}\left\langle \frac{\partial}{\partial x},\frac{\partial}{\partial x}\right\rangle=-\frac{1}{2} \frac{\partial}{\partial t}(h^2)=-hh^{\prime}.
\end{align}
Hence at $P$ we find that
\begin{equation}\label{21}0=\Delta_f\phi=\frac{\partial^2 \phi}{\partial t^2}-k_g\frac{\partial \phi}{\partial t}+\frac{\partial^2 \phi }{\partial x}-\frac{\partial f}{\partial t}\frac{\partial \phi}{\partial t}-
\frac{\partial f}{\partial x}\frac{\partial \phi}{\partial x}.
\end{equation}
Using the equality $(\ref{20})$ we get that
\begin{equation}\label{22}
\frac{\partial v}{\partial t}(P)=\frac{\partial \phi}{\partial t}\frac{\partial^2\phi }{\partial t^2}+\frac{\partial \phi}{\partial x}\frac{\partial^2\phi}{\partial t\partial x}+k_g\left(\frac{\partial \phi}{\partial x}\right)^2.
\end{equation}
Multiplying the equation $(\ref{21})$ by $-\frac{\partial \phi}{\partial t}$ and adding with the equation 
$(\ref{22})$ we obtain 
\begin{equation}\label{23}
\frac{\partial v}{\partial t}(P)=k_g|\nabla \phi|^2-\frac{\partial \phi}{\partial t}\frac{\partial^2\phi }{\partial x^2}+
\frac{\partial\phi }{\partial x}\frac{\partial^2\phi}{\partial t\partial 
x}+\frac{\partial f}{\partial t}\left(\frac{\partial \phi}{\partial t}\right)^2+
\frac{\partial f}{\partial x}\frac{\partial \phi}{\partial x}\frac{\partial 
\phi}{\partial t}.
\end{equation}
If $\frac{\partial \phi}{\partial x}(P)\neq0$, the equation $(\ref{24})$ and the boundary condition yields
\begin{equation}\label{29}
\frac{\partial^2 \phi}{\partial x^2}(P)=p_1\frac{\partial \phi}{\partial t}(P).
\end{equation}
Therefore the equation $(\ref{23})$ can be re-written using the boundary 
condition as 
\begin{equation}\label{30}
\frac{\partial v}{\partial t}(P)=(k_g-p_1)|\nabla 
\phi|^2+p_1\left(\frac{\partial \phi}{\partial x}\right)^2+
\frac{\partial\phi }{\partial x}\frac{\partial^2\phi}{\partial t\partial 
x}+\frac{\partial f}{\partial t}\left(\frac{\partial \phi}{\partial t}\right)^2+
\frac{\partial f}{\partial x}\frac{\partial \phi}{\partial x}\frac{\partial 
\phi}{\partial t}.
\end{equation}
Notice that by $(\ref{24})$ we obtain, using $(\ref{29})$, 
\begin{equation}
 0=\frac{\partial \phi}{\partial t}\frac{\partial^2\phi }{\partial x\partial 
t}+\frac{\partial \phi}
{\partial x}\frac{\partial^2\phi }{\partial x^2}=\frac{\partial \phi}{\partial 
t}\left(\frac{\partial^2\phi }{\partial x\partial 
t}+p_1\frac{\partial \phi}
{\partial x}\right),
\end{equation}
that is,
$$
p_1\frac{\partial \phi}{\partial 
x}=-\frac{\partial^2\phi }{\partial x\partial 
t}.
$$
Thus $(\ref{30})$ becomes
\begin{align*} 
\frac{\partial v}{\partial t}(P)&=(k_g-p_1)|\nabla 
\phi|^2+\frac{\partial 
f}{\partial t}\left(\frac{\partial \phi}{\partial t}\right)^2+
\frac{\partial f}{\partial x}\frac{\partial \phi}{\partial x}\frac{\partial 
\phi}{\partial t}
\end{align*}
and we write
$$
\frac{\partial v}{\partial t}(P)=(k_g-p_1)|\nabla 
\phi|^2+\frac{\partial \phi}{\partial t}\langle \nabla\phi, \nabla f\rangle.
$$
Since $f|_{\partial M}$ is constant, so $\frac{\partial f}{\partial x}(P)=0$, 
and using that $\nu$ coincide with normalized gradient of $f$ in $\partial M$ we have that $\frac{\partial 
f}{\partial t}\leq0$
\begin{align*} 
\frac{\partial v}{\partial t}(P)&=(k_g-p_1)|\nabla 
\phi|^2+\frac{\partial 
f}{\partial t}\left(\frac{\partial \phi}{\partial t}\right)^2\\
&\geq\left(k_g+\frac{\partial 
f}{\partial t}-p_1\right)|\nabla \phi|^2,
\end{align*}
hence
\begin{equation}
(k_g+\frac{\partial 
f}{\partial t}-p_1)|\nabla \phi|^2<0,
\end{equation}
and $p_1>k_g+\frac{\partial 
f}{\partial t}=k_g-f_{\nu}\geq c$.

Now we assume that $\frac{\partial \phi}{\partial x}(P)=0$. A straighforward 
calculation yields 
$$\frac{\partial^2 v}{\partial x^2}(P)=\left(\frac{\partial^2\phi}{\partial 
x\partial t}\right)^2+\frac{\partial \phi}{\partial 
t}\frac{\partial^3\phi}{\partial x^2\partial 
t}+\left(\frac{\partial^2\phi}{\partial 
x^2}\right)^2.$$
Using the boundary condition we get that 
\begin{equation}\label{31}
\frac{\partial^2 v}{\partial 
x^2}(P)=p_1^2\phi\frac{\partial^2\phi}{\partial 
x^2}+\left(\frac{\partial^2\phi}{\partial x^2}\right)^2\leq0.
\end{equation}
Since $\frac{\partial \phi}{\partial x}(P)=0$, the equation $(\ref{23})$ 
implies that
$$\frac{\partial v}{\partial 
t}(P)=k_g\left(\frac{\partial \phi}{\partial 
t}\right)^2+p_1\phi\frac{\partial^2\phi}{\partial 
x^2}+\frac{\partial f}{\partial t}\left(\frac{\partial \phi}{\partial 
t}\right)^2 =\left(k_g+\frac{\partial f}{\partial 
t}\right)p_1^2\phi^2+p_1\phi\frac{\partial^2\phi}{\partial x^2}.$$
Thus 
\begin{equation}\label{32}
\left(k_g+\frac{\partial f}{\partial 
t}\right)p_1^3\phi^2+p_1^2\phi\frac{\partial^2\phi}{\partial x^2}<0.
\end{equation}
Adding inequality $(\ref{31})$ with $(\ref{32})$ we obtain
$$\left(\frac{\partial^2\phi}{\partial 
x^2}+p_1^2\phi\right)^2+p_1^3\left(k_g+\frac{\partial f}{\partial 
t}-p_1\right)\phi^2<0.$$
Hence 
$$p_1>k_g+\frac{\partial f}{\partial t}=k_g-f_{\nu}\geq c.$$

Let's assume that $v$ is the constant function. Observe that $v\notequiv0$ 
because $\phi$ is non-constant. Since $v$ is $f$-harmonic, inequality 
$(\ref{33})$ implies that
$$\Hess\phi=0\,\,\,\,\,\,\text{ and }\,\,\,\,\,\, \Ric_f(\nabla\phi,\nabla\phi)=0,\,\,\,\,\,\text{ on 
}\,\,\,\,M.$$

Now, using that $\Delta_f\phi=0$, we obtain that $\langle \nabla \phi, \nabla f\rangle=0$ and hereby $\Hess f(\nabla\phi, \nabla\phi)=0.$ Thus, the Gaussian curvature $K$ of $M$ vanishes. Moreover, using the structure of surfaces,  \begin{equation}\label{854}\nabla f=\lambda\, J(\nabla\phi),\end{equation} where $J$ is the anti-clockwise rotation of $\pi/2$  in the tangent plane.

Let $\{e_1,e_2\}$ be a local orthonormal frame field such that $e_1$ is tangent 
to $\partial M$ and $e_2=\eta$. So,
\begin{align*}
 0=\Hess\phi(e_1,e_2)&=e_1e_2(\phi)-\nabla_{e_1}e_2(\phi)\\
 &=e_1(p_1\phi)-\langle\nabla_{e_1}e_2,e_1\rangle\phi_1\\
 &=(p_1-k_g)\phi_1.
\end{align*}
Observe that if $\phi_1=0$ on $\partial M$, then $\phi=$constant on $\partial 
M$ and hence $\phi$ is a constant function on $M$ which is a contradiction. 
Thus $p_1=k_g$ except maybe when $\phi_1=0$. Since 
$\Hess\phi(e_1,e_1)=0$ we have
\begin{align*}
 0=\Hess\phi(e_1,e_1)&=e_1e_1(\phi)-\nabla_{e_1}e_1(\phi)\\
 &=e_1(e_1\phi)-\langle\nabla_{e_1}e_1,e_2\rangle e_2(\phi)\\
 &=e_1(e_1\phi)+k_gp_1\phi.
\end{align*}
Hence $\phi$ satisfies on the boundary a second order differential equation 
\begin{align}\label{35}
 \frac{d^2\phi}{dx^2}+k_gp_1\phi&=0\\
 \phi(0)&=\phi(\ell)\nonumber
\end{align}
where $\ell$ represents the length of $\partial M$. The function $\phi$ does 
not vanishes identically, thus $\phi_1=0$ except for a finite number of points. 
Therefore $p_1=k_g$ except for a finite number of points and using the 
continuity of $k_g$, we conclude that $p_1=k_g$ everywhere. Therefore,
$$p_1 = k_g-f_\nu + f_\nu\geq c,$$
and the equality between $p_1$ and $c$ occurs if  $k_g=k_0$ and $f_\nu=0$. Using 
$K=0$ and $k_g$ is a positive constant, we conclude that $M$ is an Euclidean ball.

Furthermore, the identity $(\ref{854})$ and after a straightforward computations we obtain that
$$\Hess f = \dfrac{\Delta f}{2v}\left( J(\nabla\phi)\otimes J(\nabla\phi)\right).$$ It easy to see, 
using that $M$ is an Euclidean ball, that $\phi=x_i$, that is, $\phi$ is a coordinate function.  Thus, using the expression of $\phi$,  $f$ satisfies $\Hess f = 0$ and as $f$ is constant on the boundary, we have $f$  constant.


\qed

\bigskip

\end{document}